\documentclass[11pt]{amsart}
\usepackage{amssymb,amsmath,amsthm,newlfont,enumerate}
\newtheorem{prop}{Proposition}[section]
\newtheorem{thm}[prop]{Theorem}
\newtheorem{lem}[prop]{Lemma}
\newtheorem{cor}[prop]{Corollary}
\theoremstyle{remark}

\newtheorem*{rems}{Remarks}

\newcommand{\cB}{{\cal B}}
\newcommand{\cD}{{\cal D}}
\newcommand{\cN}{{\cal N}}
\newcommand{\CC}{{\mathbb C}}
\newcommand{\DD}{{\mathbb D}}
\newcommand{\RR}{{\mathbb R}}
\newcommand{\TT}{{\mathbb T}}

\newcommand{\lt}{<}
\newcommand{\gt}{>}

\renewcommand{\tilde}{\widetilde}

\begin{document}

\title[Cyclicity in the Dirichlet space]
{On the Brown--Shields conjecture for cyclicity in the Dirichlet space}

\author[El-Fallah]{Omar El-Fallah}
\address{D\'{e}partement de Math\'ematiques\\
\\Universit\'e Mohamed V\\B.\ P.\ 1014 Rabat\\Morocco}
\email{elfallah@fsr.ac.ma}

\author[Kellay]{Karim Kellay}
\address{CMI\\LATP\\Universit\'e de Provence\\
39, rue F. Joliot-Curie\\13453 Marseille cedex 13\\France}
\email{kellay@cmi.univ-mrs.fr}

\author[Ransford]{Thomas Ransford}
\address{D\'{e}partement de math\'{e}matiques et de statistique
\\Universit\'{e} Laval\\Qu\'{e}bec (QC) \\ Can\-a\-da G1V 0A6}
\email{ransford@mat.ulaval.ca}

\thanks{The research of the second author was supported in part by ANR Dynop.
The research of the third author was partially supported 
by grants from NSERC (Canada), FQRNT (Qu\'ebec), and the Canada research chairs program.} 



\begin{abstract}
Let $\cD$ be the Dirichlet space, 
namely the space of holomorphic functions on the unit disk
whose  derivative is square-integrable. 
We establish a new sufficient condition for a function
$f\in\cD$ to be {\em  cyclic}, 
i.e.\ for $\{pf: p\text{ a polynomial}\}$ to be dense in $\cD$. 
This allows us to prove a special case of the conjecture of Brown and Shields
that a function is cyclic in $\cD$ iff it is outer and its zero set 
(defined appropriately) is of capacity zero.
\end{abstract}

\date{17 September 2008}

\maketitle

\section{Introduction}\label{S:intro}

Let $f$ be a holomorphic function on the unit disk $\DD$. 
The {\em Dirichlet integral} of $f$ is defined by
$$
\cD(f):=\frac{1}{\pi}\iint_\DD |f'(z)|^2\,dx\,dy.
$$
The {\em Dirichlet space} $\cD$ 
is the space of  holomorphic functions
$f$ on $\DD$ such that $\cD(f)\lt\infty$. 
It becomes a Hilbert space under the norm $\|\cdot\|_\cD$ defined by
$\|f\|_\cD^2:=\|f\|_{H^2}^2+\cD(f)$.

A subspace $M$ of $\cD$ is called \textit{invariant} 
if $f(z)\in M\implies zf(z)\in M$.
Given $f\in\cD$, we denote by $[f]$ 
the smallest closed, invariant subspace containing $f$,
namely the closure in $\cD$ of the set $\{pf:p\text{~is a polynomial}\}$.
We say that $f$ is \textit{cyclic} if $[f]=\cD$.  

It is a long-standing open problem to characterize the cyclic functions in $\cD$. 
Brown and Shields  showed in \cite{BS} that, if $f\in\cD$ is cyclic, 
then necessarily $f$ is an outer function  and 
the set $\{\zeta\in\TT:\lim_{r\to1^{-}}f(r\zeta)=0\}$ 
is  of logarithmic capacity zero.
They  conjectured that these two necessary conditions for cyclicity 
are between them also sufficient \cite[Question~12]{BS}. 
The article \cite{EKR} contains a brief history of the 
progress made towards proving this conjecture, 
and we shall have more to say about this at the end of the paper.

Given  $E\subset\TT$ and $t\ge0$, 
we write $E_t:=\{\zeta\in\TT:d(\zeta,E)\le t\}$, 
where $d$ denotes arclength distance on the unit circle $\TT$. 
Also, we write $|E_t|$ for the Lebesgue measure of $E_t$. 
The following theorem is our main  result.
 
\begin{thm}\label{T:main}
Let $f\in\cD$ be an outer function, and set
$E:=\{\zeta\in\TT:\liminf_{z\to \zeta}|f(z)|=0\}$. 
Suppose that
$|E_t|=O(t^\mu)$ as $t\to0$
for some $\mu\gt0$, and that
\begin{equation}\label{E:capcond}
\int_0^\pi\frac{dt}{|E_t|}=\infty.
\end{equation}
Then $f$ is cyclic in $\cD$.
\end{thm}

\begin{rems}
(i) A compact set $E\subset\TT$ satisfying  the condition \eqref{E:capcond} 
automatically has capacity zero.
This follows, for example, from \cite[\S IV, Theorem~2]{Ca3}.

(ii) For certain types of set, condition \eqref{E:capcond} 
is actually equivalent to capacity zero.
Let $(l_n)_{n\ge0}$ be a sequence in $(0,2\pi)$ 
such that $\lambda:=\sup_{n\ge0}{l_{n+1}}/l_n\lt1/2$, 
and let $E$ be the associated generalized Cantor set. 
(Thus, we begin with a closed arc of length
$l_0$, remove an open arc from the middle to leave two closed arcs of length $l_1$, 
remove open arcs from their middles to
leave four arcs of length $l_2$, etc.; 
then $E$ is the intersection of the resulting nested sequence of sets.)
Then \eqref{E:capcond} holds if and only if $E$ is of capacity zero:
see for example \cite[\S IV, Theorem~3]{Ca3} and its proof.
Moreover, it is easy to see that $|E_t|=O(t^\mu)$ as $t\to0$, 
where $\mu=1-\log 2/\log(1/\lambda)$. 
Thus we deduce the following result, which proves a special case of the
Brown--Shields conjecture.
\end{rems}

\begin{cor}\label{C:BS}
Let $f\in\cD$. 
Assume that $|f|$ extends continuously to $\overline{\DD}$ 
and that $E:=\{\zeta\in\TT:|f(\zeta)|=0\}$ is a generalized Cantor set 
in the sense defined above. 
Then $f$ is cyclic if and only if $f$ is outer and $E$ is of capacity zero.
\end{cor}

\begin{proof}
In view of the remarks above, the sufficiency follows from Theorem~\ref{T:main}.
Necessity comes from the results of Brown and Shields \cite[Theorem~5]{BS}.
\end{proof}

The rest of the paper is devoted to the proof of Theorem~\ref{T:main}.
We begin in \S\ref{S:background} by recalling some basic background on
the Dirichlet space.
Then,  in \S\ref{S:fusion}, 
we prove a general theorem about invariant subspaces of $\cD$,
based on a technique of Korenblum and on a  fusion lemma for $\cD$.
In \S\ref{S:distance} we establish an estimate for the Dirichlet integral
of so-called distance functions,
namely outer functions $f$ whose boundary values 
$|f^*(\zeta)|$ depend only on $d(\zeta,E)$. 
In \S\ref{S:sun}, 
we prove a regularization theorem, 
related to the rising-sun lemma of F.~Riesz,
which is needed for smoothing the function $|E_t|$.
These ingredients are then combined
in \S\ref{S:proof} to complete the proof of Theorem~\ref{T:main}.
Finally, in \S\ref{S:BS}, we relate our results to previous work
in the area.

\subsection*{Acknowledgment}
The authors are grateful to Sasha Borichev 
for several very helpful discussions on this topic.

\section{Background on the Dirichlet space}\label{S:background}

In this section we briefly recall some basic notions about the Dirichlet space, and collect together a few results that will be needed in what follows.
For general facts concerning Hardy spaces, we refer to the books of
Garnett \cite{Ga} and Koosis \cite{Koo}. Results about the the Dirichlet
space will be cited in detail below. 
The article of Ross \cite{Ro} is a general survey of the Dirichlet space.

A first remark is that, if $f(z)=\sum_{k\ge0}a_kz^k$, then
$\cD(f)=\sum_{k\ge0}k|a_k|^2$.
It follows immediately that $\cD$ is a subspace of the Hardy space $H^2$. 
The inclusion map $\cD\hookrightarrow H^2$ is compact with dense range.

Given a holomorphic function $f$ on $\DD$ and $\zeta\in\TT$, 
we write $f^*(\zeta):=\lim_{r\to1^{-}}f(r\zeta)$
whenever this limit exists. 
We say that $f$ is {\em inner} if $f$ is bounded and $|f^*|=1$ a.e.\ on $\TT$. 
We say that $f$ is {\em outer} if it is of the form
$$
f(z)=\exp\Bigl(\frac{1}{2\pi}\int_\TT \frac{\zeta+z}{\zeta-z}\log \phi(\zeta)\,|d\zeta|\Bigr)
\qquad(z\in\DD),
$$
where $\phi$ is a positive function such that $\log\phi\in L^1(\TT)$. 
In this case $f^*$ exists a.e.\ on $\TT$ and $|f^*|=\phi$ a.e.\ on $\TT$.
Every $f\in H^2$ factorizes uniquely as $f=f_if_o$, 
where $f_i,f_o\in H^2$ with $f_i$ inner and $f_o$ outer. 

\begin{thm}\label{T:Dio}
Let $f\in\cD$ and let $f=f_if_o$ be the inner-outer factorization of $f$.
Then $f_o\in\cD$,
and $\cD(f_o)\le\cD(f)$ (but $f_i\notin\cD$ in general).
\end{thm}

\begin{proof}
See \cite{Ca2}.
\end{proof}

We shall make extensive use of the following formula of Carleson \cite{Ca2}. 

\begin{thm}\label{T:Carleson1}
Let $f$ be an outer function. Then
\begin{equation}\label{E:Carleson1}
\cD(f)=\frac{1}{4\pi^2}\iint_{\TT^2}
\frac{(|f^*(\zeta_1)|^2-|f^*(\zeta_2)|^2)
(\log|f^*(\zeta_1)|-\log|f^*(\zeta_2)|)}
{|\zeta_1-\zeta_2|^2}\,|d\zeta_1|\,|d\zeta_2|.
\end{equation}
\end{thm}

\begin{proof}
See \cite{Ca2}.
The way it is stated in \cite{Ca2}, 
the formula presupposes that $\cD(f)\lt\infty$.
However, the proof
shows that the formula holds even when $\cD(f)=\infty$.
Another proof can be found in \cite{RS1}.
\end{proof}

Recall from the introduction that, given $f\in\cD$, 
we write $[f]$ to denote the closed invariant subspace of $\cD$ generated by $f$. 
The remaining results in this section are all due to Richter and Sundberg \cite{RS1,RS2}.

\begin{thm}\label{T:RS1}
Let $f_1,f_2\in\cD$.
\begin{itemize}
\item[(i)] If $|f_1|\le|f_2|$ on $\DD$, then $[f_1]\subset[f_2]$.
\item[(ii)] If $|f_1^*|\le |f_2^*|$ a.e.\ on $\TT$ and $f_2$ is outer, 
then $[f_1]\subset[f_2]$.
\end{itemize}
\end{thm}

\begin{proof}
Part~(i) is \cite[Corollary~5.5]{RS1}.
Part~(ii) is a simple consequence of (i).
\end{proof}

\begin{thm}\label{T:RS2a}
Let $f_1,f_2\in\cD$ be outer functions and let 
$f$ be the outer function given by $|f^*|:=\min\{|f_1^*|,|f_2^*|\}$ a.e.
Then $f\in\cD$ and 
$[f]=[f_1]\cap [f_2]$.
If further $f_1f_2\in\cD$, then $[f_1f_2]=[f_1]\cap[f_2]$.
\end{thm}

\begin{proof}
See \cite[Lemma~2.2 and Theorem~4.5]{RS2}.
\end{proof}

\begin{thm}\label{T:RS2b}
Let $f$ be an outer function and let $\alpha\gt0$. 
Suppose that both $f,f^\alpha\in\cD$.
Then $[f^\alpha]=[f]$.
\end{thm}

\begin{proof}
See \cite[Theorem~4.3]{RS2}.
\end{proof}

\section{Korenblum's method and the fusion lemma}\label{S:fusion}

The first step towards proving Theorem~\ref{T:main} is the following theorem.

\begin{thm}\label{T:Korenblum}
Let $f\in\cD$ be an outer function, and define 
$E:=\{\zeta\in\TT:\liminf_{z\to\zeta}|f(z)|=0\}$.
If $g\in\cD$ and
$|g^*(\zeta)|\le  d(\zeta,E)$ a.e.\ on $\TT$, 
then $g\in[f]$.
\end{thm}

A slightly weaker result along these lines was already implicit in \cite{EKR}
(see Theorems~2.1 and 3.1 in that paper). 
There it was a consequence of the so-called resolvent method of Carleman, 
as exposed for example in \cite{HS,RS3}.
The proof of Theorem~\ref{T:Korenblum} below, 
based on an adaptation of a technique due to Korenblum \cite{Kor}, 
is direct and much simpler. 

We begin with a simple closure lemma.

\begin{lem}\label{L:closure}
Let $M$ be a closed subspace of $\cD$ and let $f$ be an outer function. 
Suppose that there exists a sequence $(f_n)$ in $M$ such that:
\begin{itemize}
\item[(i)] $|f_n^*|\to|f^*|$ a.e.\ on $\TT$,
\item[(ii)] $|f_n(0)|\to|f(0)|$, 
\item[(iii)] $\sup_n\cD(f_n)\lt\infty$.
\end{itemize}
Then $f\in M$.
\end{lem}

\begin{proof}
By (ii) and (iii) together, $(f_n)$ is norm-bounded in the Hilbert space $\cD$,
so a subsequence $(f_{n_j})$ converges weakly in~$\cD$, to $g$ say.
As a closed subspace of $\cD$ is weakly closed, we have $g\in M$.
Also, we have $f_{n_j}^*\to g^*$ in $L^2(\TT)$ 
(because the inclusion $\cD\hookrightarrow H^2$ is compact)
and $f_{n_j}(0)\to g(0)$. 
From (i) and (ii), it follows that $|f^*|=|g^*|$ a.e.\ on $\TT$ and
$|f(0)|=|g(0)|$. 
As $f$ is outer, we deduce that $f=cg$ for some unimodular constant $c$.
Hence, finally, $f\in M$, as claimed.
\end{proof}

Next we prove a fusion lemma for $\cD$, 
which may be of independent interest.

\begin{lem}\label{L:fusion}
Let $E$ be a closed subset of $\TT$ of measure zero.
Let $h_1,\dots,h_n\in\cD$ be outer functions satisfying
$|h_j^*(\zeta)|\le \pi^{-1}d(\zeta,E)$ a.e. $(j=1,\dots,n)$.
Let $\TT\setminus E=U_1\cup\dots\cup U_n$ be a partition
of $\TT\setminus E$ into $n$ open subsets, 
and let $h$ be the outer function
such that $|h^*|=|h^*_j|$ on $U_j~(j=1,\dots,n)$.
Then $h\in\cD$ and
\begin{equation}\label{E:fusion}
\cD(h)\le
\sum_{j=1}^n\cD(h_j)+\frac{1}{2}
\sum_{j=1}^n\log\frac{1}{|h_j(0)|}.
\end{equation}
\end{lem}

\begin{proof}
By Carleson's formula \eqref{E:Carleson1},
\begin{align*}
\cD(h)&
=\frac{1}{4\pi^2}\iint_{\TT^2}
\frac{(|h^*(\zeta)|^2-|h^*(\zeta')|^2)
(\log|h^*(\zeta)|-\log|h^*(\zeta')|)}
{|\zeta-\zeta'|^2}\,|d\zeta|\,|d\zeta'|\\
&= 
\frac{1}{4\pi^2}\sum_{j,k}\int_{U_k}\int_{U_j}
\frac{(|h_j^*(\zeta)|^2-|h_k^*(\zeta')|^2)
(\log|h_j^*(\zeta)|-\log|h_k^*(\zeta')|)}
{|\zeta-\zeta'|^2}\,|d\zeta|\,|d\zeta'|.
\end{align*}

The terms with $j=k$ are estimated using Carleson's formula again.
For each $j$ we have
$$
\frac{1}{4\pi^2}\int_{U_j}\int_{U_j}
\frac{(|h_j^*(\zeta)|^2-|h_j^*(\zeta')|^2)
(\log|h_j^*(\zeta)|-\log|h_j^*(\zeta')|)}
{|\zeta-\zeta'|^2}\,|d\zeta|\,|d\zeta'|
\le \cD(h_j).
$$

Now suppose that $j\ne k$.
If $\zeta\in U_j$ and $\zeta'\in U_k$,
then there exists a point of $E$ between them, so  
$d(\zeta,\zeta')\ge d(\zeta,E)+d(\zeta',E)$, and consequently
$$
\Bigl|\frac{|h_j^*(\zeta)|^2-|h_k^*(\zeta')|^2}{|\zeta-\zeta'|^2}\Bigr|
\le \frac{\pi^{-2}d(\zeta,E)^2+\pi^{-2}d(\zeta',E)^2}{(4/\pi^2)d(\zeta,\zeta')^2}
\le \frac{1}{4}.
$$
Note also that the hypothesis $|h_j^*(\zeta)|\le \pi^{-1}d(\zeta,E)$ implies that
$\|h_j\|_\infty\le 1$. 
Hence, if $j\ne k$, then
\begin{align*}
\int_{U_k}\int_{U_j}
&\frac{(|h_j^*(\zeta)|^2-|h_k^*(\zeta')|^2)
(\log|h_j^*(\zeta)|-\log|h_k^*(\zeta')|)}
{|\zeta-\zeta'|^2}\,|d\zeta|\,|d\zeta'| \\
&\le \frac{1}{4}\int_{U_k}\int_{U_j}
\Bigl|\log|h_j^*(\zeta)|-\log|h_k^*(\zeta')|\Bigr|\,|d\zeta|\,|d\zeta'| \\
&\le \frac{1}{4}\int_{U_k}\int_{U_j}
\Bigl(\log\frac{1}{|h_j^*(\zeta)|}+\log\frac{1}{|h_k^*(\zeta')|}\Bigr)
\,|d\zeta|\,|d\zeta'| \\
&=\frac{1}{4}|U_k|\int_{U_j}\log\frac{1}{|h_j^*(\zeta)|}\,|d\zeta|
+\frac{1}{4}|U_j|\int_{U_k}\log\frac{1}{|h_k^*(\zeta')|}\,|d\zeta'|.
\end{align*}
Therefore,
\begin{align*}
\frac{1}{4\pi^2}\sum_{\substack{j,k\\j\ne k}}
&\int_{U_k}\int_{U_j}
\frac{(|h_j^*(\zeta)|^2-|h_k^*(\zeta')|^2)
(\log|h_j^*(\zeta)|-\log|h_k^*(\zeta')|)}
{|\zeta-\zeta'|^2}\,|d\zeta|\,|d\zeta'| \\
&\le\frac{1}{4\pi}\sum_j\int_\TT\log\frac{1}{|h_j^*(\zeta)|}\,|d\zeta|
=\frac{1}{2}\sum_j\log\frac{1}{|h_j(0)|},
\end{align*}
the last equality because each $h_j$ is outer. 

Finally, combining these estimates, we see that \eqref{E:fusion}
holds, and  the proof is complete.
\end{proof}

We now turn to the proof of Theorem~\ref{T:Korenblum}.
As noted earlier, it is based upon an technique due to Korenblum \cite{Kor}.
Further applications of this technique may be found in \cite{Bo,Ma}.
In the course of the proof, we shall use Lemma~\ref{L:fusion} 
several times, always with $n=2$.
What  is important is that the estimate \eqref{E:fusion} 
depends only on $E,h_1,h_2$ and not on the choice of partition $U_1,U_2$.

\begin{proof}[Proof of Theorem~\ref{T:Korenblum}]
Let $f_1$ be the outer function such that 
$|f_1^*|=\min\{|f^*|,1\}$ a.e.\ on $\TT$.
Then by Theorem~\ref{T:RS2a}, we have $f_1\in\cD$ 
and $[f_1]=[f]$. Thus there is no loss of
generality in supposing, from the  outset, that $|f|\le1$. 

Also,
if $g=g_ig_o$ is the inner-outer factorization of $g$, 
then $|g_o^*|=|g^*|$ a.e.\ and by Theorem~\ref{T:RS1} we have $g\in[g_o]$. 
Thus, without loss of generality, we may suppose that $g$ is outer.

Let $I_1,I_2,\dots$ be the connected components of $\TT\setminus E$.
For each $n\ge1$,  let $g_n$ be the outer function such
that 
$$
|g_n^*(\zeta)|=
\begin{cases}
|g^*(\zeta)|, &\zeta\in\cup_{j\le n} I_j\\
|g^*(\zeta)f^*(\zeta)|, &\zeta\in\cup_{j>n}I_j.
\end{cases}
$$
We claim that:
\begin{itemize}
\item[(i)]
$|g_n^*|\to |g^*|$ a.e.,
\item[(ii)]
$|g_n(0)|\to|g(0)|$,
\item[(iii)]
$\sup_n\cD(g_n)\lt\infty$,
\item[(iv)]
$g_n\in[f]$ for all $n$.
\end{itemize}
If so, then by Lemma~\ref{L:closure} we have $g\in[f]$, as desired.

It is obvious that (i) and (ii) hold. Also (iii) follows from Lemma~\ref{L:fusion},
applied with $h_1:=g/\pi$ and $h_2:=gf/\pi$. 
It remains to prove (iv). For this, consider
first $I_1=(e^{ia},e^{ib})$. 
Choose $a_k\downarrow a$ and $b_k\uparrow b$. For each $k$, 
let $\phi_k$ be the outer function such that
$$
|\phi_k^*(\zeta)|=
\begin{cases}
|(\zeta-e^{ia_k})(\zeta-e^{ib_k})g^*(\zeta)|, 
&\zeta\in(e^{ia_k},e^{ib_k})\\
|(\zeta-e^{ia_k})(\zeta-e^{ib_k})g^*(\zeta)f^*(\zeta)|, 
&\zeta\notin[e^{ia_k},e^{ib_k}].
\end{cases}
$$
Clearly $|\phi_k^*(\zeta)|\to |(\zeta-e^{ia})(\zeta-e^{ib})g_1^*(\zeta)|$ a.e.\ 
and $|\phi_k(0)|\to |g_1(0)|$ as $k\to\infty$.
By Lemma~\ref{L:fusion} again, $\sup_k\cD(\phi_k)\lt\infty$. 
Also, from the way that $E$ is defined,
each function $|\phi_k^*/f^*|$ is bounded on $\TT$, 
so  using Theorem~\ref{T:RS1} we have $\phi_k\in[f]$.
By Lemma~\ref{L:closure}, we deduce that $(z-e^{ia})(z-e^{ib})g_1\in[f]$.
But also, by Theorem~\ref{T:RS2a},
$$
[(z-e^{ia})(z-e^{ib}b)g_1]=[(z-e^{ia})]\cap [(z-e^{ib})]\cap[g_1]=[g_1],
$$
the last equality because $(z-e^{ia})$ and $(z-e^{ib})$ 
are both cyclic in $\cD$ (see e.g.\ \cite[Lemma~8]{BS}).
Hence $g_1\in[f]$. An obvious adaptation of this argument shows that
$g_n\in[f]$ for each $n$, giving (iv) above, and thus completing the proof.
\end{proof}

\section{Distance functions}\label{S:distance}

Let $E$ be a closed subset of $\TT$ of Lebesgue measure zero, and 
let $w:(0,\pi]\to\RR^+$ be a continuous function such that
\begin{equation}\label{E:log}
\int_\TT |\log w(d(\zeta,E))|\,|d\zeta|\lt\infty.
\end{equation}
We shall denote by $f_w$  the outer function given by
\begin{equation}\label{E:fw}
|f_w^*(\zeta)|=w(d(\zeta,E)) \quad\text{a.e.}
\end{equation}
Functions of this kind were already studied, for example,
by Carleson in \cite{Ca1}, in the course of his construction
of outer functions in $A^k(\DD)$ with prescribed zero sets.
(Here $A^k(\overline{\DD})$ is the space of
$f\in C^k(\overline{\DD})$ that are holomorphic on $\DD$.)
As the functions $f_w$ do not seem to bear a special name,
we have christened them {\em distance functions}.
Our basic result is a two-sided estimate for the Dirichlet integral of 
certain distance functions.

\begin{thm}\label{T:Dint}
Let $E$ be a closed subset of $\TT$ of measure zero, 
let $w:(0,\pi]\to\RR^+$ be an increasing function such that
\eqref{E:log} holds,
and let $f_w$ be the outer function given by \eqref{E:fw}. 
Suppose further that there exists $\gamma\gt2$ such that
$t\mapsto w(t^\gamma)$ is concave. Then
\begin{equation}\label{E:Dint1}
\cD(f_w)\asymp \int_\TT w'(d(\zeta,E))^2d(\zeta,E)\,|d\zeta|,
\end{equation}
where the implied constants depend only on $\gamma$.
In particular, $f_w\in\cD$ iff the integral in \eqref{E:Dint1} is finite.
\end{thm}

Before going on, 
it will be convenient to introduce a little more notation.
Given  a closed subset $E$ of $\TT$ of Lebesgue measure zero, 
we write
\begin{equation}\label{E:NE}
N_E(t):=2\sum_j 1_{\{|I_j|\gt 2t\}}
\qquad(0\lt t\le\pi),
\end{equation}
where $(I_j)$ are the components of $\TT\setminus E$,
and $|\cdot|$ denotes Lebesgue measure on $\TT$. 
It is then elementary to check that, for every  measurable function
$\Omega:(0,\pi]\to\RR^+$, 
\begin{equation}\label{E:N_E}
\int_\TT\Omega(d(\zeta,E))\,|d\zeta|
=\int_0^\pi \Omega(t)N_E(t)\,dt.
\end{equation}
For example, taking $\Omega(t):=1_{[0,\delta]}$, 
we have $\int_0^\delta N_E(t)\,dt=|E_\delta|$,
where $E_\delta:=\{\zeta\in\TT:d(\zeta,E)\le\delta\}$.
In particular $\delta N_E(\delta)\le |E_\delta|$.
Note also that, in this notation, \eqref{E:log} is equivalent to
\begin{equation}\label{E:log2} 
\int_0^\pi|\log w(t)|N_E(t)\,dt\lt\infty,
\end{equation} 
and \eqref{E:Dint1} now becomes
\begin{equation}\label{E:Dint2}
\cD(f_w)\asymp \int_0^\pi w'(t)^2t N_E(t)\,dt.
\end{equation}

\begin{proof}[Proof of Theorem~\ref{T:Dint}]
In what follows, $\zeta_1,\zeta_2$ denote points of $\TT$, 
and we write $\delta_j:=d(\zeta_j,E)$.
Note that $|\delta_1-\delta_2|\le d(\zeta_1,\zeta_2)$.
In this notation, Carleson's formula \eqref{E:Carleson1} becomes
\begin{equation}\label{E:Carleson2}
\cD(f_w)
=
\frac{1}{4\pi^2}\iint_{\TT^2}
\frac{(w^2(\delta_1)-w^2(\delta_2))(\log w(\delta_1)-\log w(\delta_2))}
{|\zeta_1-\zeta_2|^2}\,|d\zeta_1|\,|d\zeta_2|.
\end{equation}
For convenience, we shall extend $w$ to the whole of $\RR^+$ by defining
$w(t):=w(\pi)$ for $t\gt\pi$.

We first establish the upper bound in \eqref{E:Dint2}. 
Starting from \eqref{E:Carleson2}, we have
\begin{align*}
\cD(f_w)
&\le 
\frac{1}{16}\iint_{\TT^2}
\frac{(w^2(\delta_1)-w^2(\delta_2))(\log w(\delta_1)-\log w(\delta_2))}
{d(\zeta_1,\zeta_2)^2}\,|d\zeta_1|\,|d\zeta_2|\\
&\le 
\frac{1}{8}\iint_{\delta_1\ge\delta_2}
\frac{(w^2(\delta_1)-w^2(\delta_2))(\log w(\delta_1)-\log w(\delta_2))}
{d(\zeta_1,\zeta_2)^2}\,|d\zeta_1|\,|d\zeta_2|\\
&\le
\frac{1}{8}\iint_{\delta_1\ge\delta_2}
\frac{(w^2(\delta_2+d(\zeta_1,\zeta_2))-w^2(\delta_2))
(\log w(\delta_2+d(\zeta_1,\zeta_2))-\log w(\delta_2))}
{d(\zeta_1,\zeta_2)^2}\,|d\zeta_1|\,|d\zeta_2|\\
&\le
\frac{1}{4}\int_{\TT}\int_{0}^\pi
\frac{(w^2(\delta_2+s)-w^2(\delta_2))
(\log w(\delta_2+s)-\log w(\delta_2))}
{s^2}\,ds\,|d\zeta_2|\\
&=
\frac{1}{4}\int_{0}^\pi\int_{0}^\pi
\frac{(w^2(t+s)-w^2(t))
(\log w(t+s)-\log w(t))}
{s^2}\,ds\,N_E(t)\,dt.
\end{align*}
To estimate this, we now exploit the concavity assumption on $w$.
This assumption amounts to saying that $t\mapsto w'(t) t^{1-1/\gamma}$ 
is decreasing. Thus
\begin{align*}
w^2(t+s)-w^2(t)
&=\int_t^{t+s} 2w(u)w'(u)\,du\\
&\le \int_t^{t+s} 2w(t+s)w'(t)(t/u)^{1-1/\gamma}\,du\\
&=2\gamma w(t+s)w'(t)t\bigl((1+s/t)^{1/\gamma}-1\bigr).
\end{align*}
Also, using the fact that $w(t)/t^{1/\gamma}$ is decreasing, we have
\begin{align*}
\log w(t+s)-\log w(t)
&=\int_t^{t+s}\frac{w'(u)}{w(u)}\,du\\
&=\int_t^{t+s}u^{1-1/\gamma}w'(u)\frac{u^{1/\gamma}}{w(u)}\,\frac{du}{u}\\
&\le\int_t^{t+s}t^{1-1/\gamma}w'(t)\frac{(t+s)^{1/\gamma}}{w(t+s)}\,\frac{du}{u}\\
&=tw'(t)\frac{(1+s/t)^{1/\gamma}}{w(t+s)}\log(1+s/t).
\end{align*}
Combining these estimates, we obtain
\begin{align*}
\int_0^\pi&\frac{(w^2(t+s)-w^2(t))(\log w(t+s)-\log w(t))}{s^2}\,ds\\
&\le \int_0^\pi 2\gamma w'(t)^2t^2\bigl((1+s/t)^{1/\gamma}-1\bigr)(1+s/t)^{1/\gamma}\log(1+s/t)\,\frac{ds}{s^2}\\
&= w'(t)^2t\int_0^{\pi/t} 2\gamma \bigl((1+x)^{1/\gamma}-1\bigr)(1+x)^{1/\gamma}\log (1+x)\,\frac{dx}{x^2}\\
&\le A_\gamma w'(t)^2t,
\end{align*}
where $A_\gamma$ is a constant depending only on $\gamma$
(here we used the fact that $\gamma\gt2$).
Plugging this into the estimate for $\cD(f_w)$ yields 
$$
\cD(f_w)\le \frac{A_\gamma}{4}\int_0^\pi w'(t)^2tN_E(t)\,dt,
$$
giving the upper bound in \eqref{E:Dint2}.

For the lower bound, 
we start once again from Carleson's formula \eqref{E:Carleson2}.
We have
\begin{align*}
\cD(f_w)
&\ge 
\frac{1}{4\pi^2}\iint_{\TT^2}
\frac{(w^2(\delta_1)-w^2(\delta_2))(\log w(\delta_1)-\log w(\delta_2))}
{d(\zeta_1,\zeta_2)^2}\,|d\zeta_1|\,|d\zeta_2|\\
&\ge
\frac{1}{4\pi^2}\iint_{\substack{\delta_1\gt\delta_2\\d(\zeta_1,\zeta_2)\lt\delta_1/2}}
\frac{(w^2(\delta_1)-w^2(\delta_2))(\log w(\delta_1)-\log w(\delta_2))}
{d(\zeta_1,\zeta_2)^2}\,|d\zeta_1|\,|d\zeta_2|\\
&\ge
\frac{1}{4\pi^2}\iint_{\substack{\delta_1\gt\delta_2\\d(\zeta_1,\zeta_2)\lt\delta_1/2}}
\frac{(w^2(\delta_1)-w^2(\delta_1/2))(\log w(\delta_1)-\log w(\delta_1/2))}
{(\delta_1/2)^2}\,|d\zeta_1|\,|d\zeta_2|.
\end{align*}
For a fixed $\zeta_1\in\TT\setminus E$, 
the set of $\zeta_2\in\TT$ satisfying
$\delta_1\gt\delta_2$ and $d(\zeta_1,\zeta_2)\lt\delta_1/2$
is an arc of length $\delta_1/2$. Hence
\begin{align*}
\cD(f_w)
&\ge\frac{1}{4\pi^2}
\int_\TT\frac{(w^2(\delta_1)-w^2(\delta_1/2))(\log w(\delta_1)-\log w(\delta_1/2))}
{\delta_1/2}\,|d\zeta_1|\\
&=\frac{1}{4\pi^2}
\int_0^\pi\frac{(w^2(t)-w^2(t/2))(\log w(t)-\log w(t/2))}{t/2}N_E(t)\,dt.
\end{align*}
Now, using the concavity property of $w$ once again, we have
$$
w^2(t)-w^2(t/2)
\ge w(t)\int_{t/2}^t w'(u)\,du
\ge w(t)\int_{t/2}^t w'(t)(t/u)^{1-1/\gamma}\,du
=B_\gamma w(t)w'(t) t,
$$
where $B_\gamma\gt0$ is a constant depending on $\gamma$.
Also, 
$$
\log w(t)-\log w(t/2)=-\frac{1}{2}\log\frac{w^2(t/2)}{w^2(t)}
\ge \frac{w^2(t)-w^2(t/2)}{2w^2(t)}.
$$
Substituting this into the estimate for $\cD(f_w)$ yields
$$
\cD(f_w)\ge \frac{B_\gamma^2}{4\pi^2}\int_0^\pi w'(t)^2tN_E(t)\,dt,
$$
which gives the lower bound in \eqref{E:Dint2}.
\end{proof}

\begin{rems}
(i) Almost the same proof works if we assume that $w$ is decreasing instead of increasing.
This can be used to obtain a sufficient condition for both $f_w$ and $f_{1/w}$ to belong to
$\cD$, in other words, for $f_w$ to be an invertible element of $\cD$. We omit the details.

(ii) The only point in the proof of the upper bound 
where we use the fact that $\gamma\gt2$ is in 
showing that
$$
\int_0^{\pi/t} 2\gamma \bigl((1+x)^{1/\gamma}-1\bigr)(1+x)^{1/\gamma}\log (1+x)\,\frac{dx}{x^2}
\le A_\gamma,
$$
a constant independent of $t$.
If, instead, $0\lt\gamma\lt 2$, 
then this integral $\asymp t^{1-2/\gamma}\log(\pi/t)$ as $t\to0$,
and we deduce that
\begin{equation}\label{E:gammalt2}
\cD(f_w)\le A_\gamma\int_0^\pi w'(t)^2 t^{2-2/\gamma}\log(\pi/t)N_E(t)\,dt,
\end{equation}
where $A_\gamma$ is a constant depending on $\gamma$. Likewise, if $\gamma=2$, then
\begin{equation}\label{E:gammaeq2}
\cD(f_w)\le A\int_0^\pi w'(t)^2 \log^2(\pi/t)N_E(t)\,dt.
\end{equation}
\end{rems}

\begin{cor}\label{C:powers}
Let $w(t)=t^\alpha$.
\begin{itemize}
\item[(i)] 
If $0\lt\alpha\lt 1/2$, 
then $f_w\in\cD\iff \int_0^\pi t^{2\alpha-1}N_E(t)\,dt\lt\infty$.
\item[(ii)]
If $\alpha\gt 1/2$, then
$f_w\in\cD\iff
\int_0^\pi \log(\pi/t)N_E(t)\,dt\lt\infty$.
\end{itemize}
\end{cor}

\begin{proof}
Part~(i) is a special case of Theorem~\ref{T:Dint}.
The sufficiency in part~(ii) follows from \eqref{E:gammalt2}.
The necessity is a consequence of \eqref{E:log2}.
\end{proof}

The appearance of the condition in (ii) is not  surprising. It is exactly
the condition of Carleson,
\begin{equation}\label{E:Ccond}
\int_\TT\log\Bigl(\frac{\pi}{d(\zeta,E)}\Bigr)\,|d\zeta|\lt\infty,
\end{equation}
characterizing the zero sets of outer functions in $A^k(\overline{\DD})$
for $k\ge1$ (see \cite{Ca1}). For this reason, closed sets $E\subset\TT$ that satisfy 
\eqref{E:Ccond} are often called {\em Carleson sets}.

\section{Regularization and the rising-sun lemma}\label{S:sun}

The third ingredient in the proof of Theorem~\ref{T:main} is the following 
regularization theorem, which will eventually be used to smooth the function
$t\mapsto|E_t|$.

\begin{thm}\label{T:reg}
Let $a\gt0$, let $\beta\in(0,1]$ and let $\phi:(0,a]\to\RR^+$ be a function such that
\begin{itemize}
\item $\phi(t)/t$ is decreasing,
\item $0\lt\phi(t)\le t^\beta$for all $t\in(0,a]$,
\item $\int_0^a dt/\phi(t)=\infty$.
\end{itemize}
Then, given $\alpha\in(0,\beta)$, there exists a function $\psi:(0,a]\to\RR^+$ such that
\begin{itemize}
\item $\psi(t)/t^\alpha$ is increasing,
\item $\phi(t)\le \psi(t)\le t^\beta$ for all $t\in(0,a]$,
\item $\int_0^a dt/\psi(t)=\infty$.
\end{itemize}
\end{thm}

The key tool in the proof of this theorem is the notion of increasing regularization.
Given a function $u:\RR^+\to\RR^+$, we define its {\em increasing regularization} 
$\tilde{u}:\RR^+\to\RR^+$ by
$$
\tilde{u}(x):=\inf\{u(y):y\ge x\} \qquad(x\in\RR^+).
$$
Clearly $\tilde{u}$ is increasing and $\tilde{u}\le u$.
Also, $\tilde{u}$ is maximal with these two properties, 
in the sense that if $v$ is any increasing function with $v\le u$ then
also $v\le \tilde{u}$.

The following result is a version of the so-called rising-sun lemma of F.~Riesz.
We prove it here in the form appropriate to our needs. 

\begin{lem}\label{L:sun}
Let $u:\RR^+\to\RR^+$ be a function that is lower semicontinuous and
right-continuous. 
Let $\tilde{u}$ be the increasing regularization of $u$ 
and set  $U:=\{x\in\RR^+:\tilde{u}(x)\lt u(x)\}$. 
Then $U$ is open in $\RR^+$. 
Further, if $a,b$ are the endpoints of any component of $U$,
then $u(a)\ge u(b)$.
\end{lem}

\begin{proof}
Let $x\in U$. Then there exists $y\gt x$ such that $u(y)\lt u(x)$. 
By lower semicontinuity
$u(y)\lt u(x')$ for all $x'$ in a neighborhood of $x$. 
All such $x'$ also belong to $U$. Thus $U$ is open in $\RR^+$.

Now let $a,b$ be the endpoints of a component of $U$. Since $U$ is open in $\RR^+$, 
we have $b\notin U$, and hence $u(y)\ge u(b)$ for all $y\ge b$.
Let $x\in(a,b)$. As $u$ is lower semicontinuous on the compact set $[x,b]$,
its minimum on this set is attained, at $x_0$ say. We then have $u(y)\ge u(x_0)$ for all
$y\ge x_0$, which implies that $\tilde{u}(x_0)=u(x_0)$ and so $x_0\notin U$. 
The only possibility is that $x_0=b$. 
Thus $u\ge u(b)$ on $[x,b]$, and in particular $u(x)\ge u(b)$.
Finally, letting $x\to a$ and using the right-continuity of $u$, we obtain $u(a)\ge u(b)$.   
\end{proof}

In the rising-sun terminology, the set $U$ corresponds to the shade. 
We shall need an estimate the proportion of $\RR^+$ that stays in the sun. 
Recall that the {\em lower density} of a Borel set $B\subset\RR^+$ is defined by
$$
\rho_-(B):=\liminf_{x\to\infty}\frac{\bigl|B\cap[0,x]\bigr|}{x}.
$$

\begin{lem}\label{L:density}
Let $u:\RR^+\to\RR^+$ be a positive function and
set $S:=\{x\in\RR^+:\tilde{u}(x)=u(x)\}$. 
Suppose that $x\mapsto u(x)-x$ is decreasing. Then $S$ is a Borel set and
$$
\rho_-(S)\ge \liminf_{x\to\infty}\frac{u(x)}{x}.
$$
\end{lem}

\begin{proof}
As $u(x)-x$ is decreasing, 
it follows that $u_1(x):=\lim_{y\downarrow x}u(y)$ exists for all~$x$.
The function $u_1$ is both lower semicontinuous and right-continuous,
and $u_1(x)-x$ is decreasing. 
Further, we have both $u_1=u$ and $\tilde{u}_1=\tilde{u}$ 
except on  countable
sets. Thus, we may as well suppose from the outset that 
$u$ is lower semicontinuous and right-continuous, 
so that Lemma~\ref{L:sun} applies.

We may also suppose that $u(x)\to\infty$ as $x\to\infty$, 
for if not, then $\liminf_{x\to\infty}u(x)/x=0$, and there is nothing to prove.
As a consequence of this supposition, $S$ is necessarily unbounded.

Let $y\in S$. 
Let $I_1,\dots, I_n$ be a finite set of components of
$U:=\RR^+\setminus S$ lying in $[0,y]$. 
We may suppose that $I_j$ has endpoints $a_j,b_j$,
where $0\le a_1\lt b_1\lt \dots\lt a_n\lt b_n\le y$. Then
$$
|I_1\cup\dots\cup I_n|=\sum_{j=1}^n(b_j-a_j)
\le \sum_{j=1}^n\bigl(b_j-u(b_j)-a_j+u(a_j)\bigr)\le y-u(y)+u(0),
$$
where, for the first inequality we used Lemma~\ref{L:sun}, and for the second
the fact that $u(x)-x$ is decreasing. As this holds for any such set of components, 
it follows that
$\bigl|U\cap[0,y]\bigr|\le y-u(y)+u(0)$. 
Recalling that $U$ is the complement of $S$, we deduce that
$$
\bigl|S\cap[0,y]\bigr|\ge u(y)-u(0) \quad(y\in S).
$$

Now, given $x\in\RR^+$, let $y$ be the smallest element of $S$ such that $y\ge x$.
Then
$$
\frac{\bigl|S\cap[0,x]\bigr|}{x}
=\frac{\bigl|S\cap[0,y]\bigr|}{x}
\ge\frac{\bigl|S\cap[0,y]\bigr|}{y}
\ge\frac{u(y)-u(0)}{y}.
$$
It follows that
$\liminf_{x\to\infty}|S\cap[0,x]|/x\ge\liminf_{y\to\infty}u(y)/y$,
thereby completing the proof.
\end{proof}

The last lemma we need is a simple fact about sets of positive lower density.

\begin{lem}\label{L:int=infty}
Let $v:\RR^+\to\RR^+$ be a positive decreasing function 
such that $\int_0^\infty v(x)\,dx=\infty$.
Then $\int_B v(x)\,dx=\infty$ for every Borel set $B\subset\RR^+$ with $\rho_-(B)\gt0$.
\end{lem}

\begin{proof}
Suppose that $\rho_-(B)\gt0$.
Then there exists $\lambda\gt0$ such that, for all sufficiently large $x$,
$$
\bigl| B\cap [0,x]\bigr|\ge \lambda x .
$$
Fix $a\gt 1/\lambda$. Then, for all sufficiently large $k$,
$$ 
\int_{B\cap[a^{k-1},a^k]}v(x)\,dx
\ge v(a^k)\bigl|B\cap[a^{k-1},a^k]\bigr|
\ge v(a^k)\Bigl(\bigl|B\cap[0,a^k]\bigr|-a^{k-1}\Bigr)
\ge v(a^k)\Bigl(\lambda a^k-a^{k-1}\Bigr).
$$
Also, for all $k$,
$$
\int_{[a^k,a^{k+1}]}v(x)\,dx\le v(a^k)(a^{k+1}-a^{k}).
$$
Hence, for all sufficiently large $k$,
$$ 
\int_{B\cap[a^{k-1},a^k]}v(x)\,dx
\ge \frac{\lambda-1/a}{a-1}\int_{[a^k,a^{k+1}]}v(x)\,dx.
$$
Summing over these $k$, we deduce that $\int_B v(x)\,dx=\infty$.
\end{proof}

\begin{proof}[Proof of Theorem~\ref{T:reg}]
By a simple change of scale, we can reduce to the case where $a=1$. 
This will simplify the notation in what follows.

Define $u:\RR^+\to\RR^+$ by the formula
$$
u(x):=-\frac{1}{1-\alpha}\log\phi(e^{-x})-\frac{\alpha}{1-\alpha}x 
\qquad(x\in\RR^+).
$$
The properties of $\phi$ are reflected in $u$ as follows:
\begin{align*}
\phi(t)/t \text{~decreasing~} 
&\quad\iff\quad 
u(x)-x \text{~decreasing},\\
\phi(t)\le t^\beta 
&\quad\iff \quad
u(x)\ge \frac{\beta-\alpha}{1-\alpha}x,\\
\int_0^1\frac{dt}{\phi(t)}=\infty
&\quad\iff\quad
 \int_0^\infty e^{(1-\alpha)(u(x)-x)}\,dx=\infty.
\end{align*}
Now let $\tilde{u}:\RR^+\to\RR^+$ be the increasing regularization of $u$,
and define $\psi:(0,1]\to\RR^+$ via the formula
$$
\tilde{u}(x):=-\frac{1}{1-\alpha}\log\psi(e^{-x})-\frac{\alpha}{1-\alpha}x 
\qquad(x\in\RR^+).
$$
The desired properties of $\psi$ correspond to properties of $\tilde{u}$ as follows:
\begin{align*}
\psi(t)/t^\alpha \text{~increasing~} 
&\quad\iff\quad 
\tilde{u}(x) \text{~increasing},\\
\phi(t)\le \psi(t)\le t^\beta 
&\quad\iff \quad
u(x)\ge \tilde{u}(x)\ge\frac{\beta-\alpha}{1-\alpha}x,\\
\int_0^1\frac{dt}{\psi(t)}=\infty
&\quad\iff\quad
 \int_0^\infty e^{(1-\alpha)(\tilde{u}(x)-x)}\,dx=\infty.
\end{align*}
It thus suffices to prove these three properties of $\tilde{u}$.
The first two are obvious. For the third, we remark that,
writing $S:=\{x\in\RR^+:\tilde{u}(x)=u(x)\}$, 
$$
\int_0^\infty e^{(1-\alpha)(\tilde{u}(x)-x)}\,dx
\ge \int_S e^{(1-\alpha)(\tilde{u}(x)-x)}\,dx
= \int_S e^{(1-\alpha)(u(x)-x)}\,dx.
$$
Also $e^{(1-\alpha)(\tilde{u}(x)-x)}$ is a decreasing function
and, by Lemma~\ref{L:density}, $\rho_-(S)\ge (\beta-\alpha)/(\alpha-1)\gt0$.
Therefore by Lemma~\ref{L:int=infty} the last integral diverges,
and the proof of the  theorem is complete.
\end{proof}

\section{Completion of the proof of Theorem~\ref{T:main}}\label{S:proof}

Let $f$ be the function in the statement of the theorem. 
Our aim is to prove that $1\in[f]$. 

Let $g$ be the outer function such that
$$
|g^*(\zeta)|=d(\zeta,E) \quad\text{a.e.}
$$
In the notation of \S\ref{S:distance}, we have $g=f_w$, where $w(t)=t$. 
Thus, by Corollary~\ref{C:powers}, $g\in\cD$ provided that
$\int_0^\pi\log(\pi/t)N_E(t)\,dt\lt\infty$.
Now $tN_E(t)\le|E_t|$, and by assumption
$|E_t|=O(t^{\mu})$ for some $\mu\gt0$,
so indeed $g\in\cD$. Theorem~\ref{T:Korenblum} therefore applies,
and we deduce that $g\in[f]$.

Next, we fix $\alpha$ with $1/2\lt\alpha\lt(1+\mu)/2$, 
and consider $g^{1-\alpha}$.
In the notation of \S\ref{S:distance}, 
we have $g^{1-\alpha}=f_w$, where now $w(t)=t^{1-\alpha}$.
By Corollary~\ref{C:powers}, $g^{1-\alpha}\in\cD$ 
provided  $\int_0^\pi t^{1-2\alpha}N_E(t)\,dt\lt\infty$, 
and by our choice of $\alpha$ this latter integral is indeed finite. 
From Theorem~\ref{T:RS2b} we have $[g^{1-\alpha}]=[g]$,
and consequently $g^{1-\alpha}\in[f]$.

The rest of the proof consists of showing that $1\in[g^{1-\alpha}]$. 
We shall achieve this by constructing a family of functions 
$w_\delta:(0,\pi]\to\RR^+$ for $0\lt\delta\lt1$,
such that the corresponding distance functions $f_{w_\delta}$
belong to $[g^{1-\alpha}]$ and satisfy:
\begin{itemize}
\item[(i)] $|f_{w_\delta}^*|\to 1$ a.e.\ as $\delta\to0$,
\item[(ii)] $|f_{w_\delta}(0)|\to 1$ as $\delta\to0$,
\item[(iii)] $\liminf_{\delta\to0}\cD(f_{w_\delta})\lt\infty$.
\end{itemize}
If such a family exists, then by Lemma~\ref{L:closure} 
we have $1\in[g^{1-\alpha}]$, as desired.

Here is the construction.
Fix $\beta$ with $\alpha\lt\beta\lt(1+\mu)/2$,
and define a function $\phi:(0,\pi]\to\RR^+$ by
$\phi(t):=\min\{|E_t|,t^\beta\}$.
This function satisfies the hypotheses of Theorem~\ref{T:reg}, 
so there exists a function $\psi:(0,\pi]\to\RR^+$ satisfying the
conclusions of that theorem, namely:
$\psi(t)/t^\alpha$ is increasing,
$\phi(t)\le\psi(t)\le t^\beta$ for all $t\in(0,\pi]$,
and $\int_0^\pi dt/\psi(t)=\infty$.
Note that, for $0\lt t\lt 1$, we have $\psi(t)\ge\phi(t)\ge t$.
For $0\lt\delta\lt1$, we define $w_\delta:(0,\pi]\to\RR^+$ by
$$
w_\delta(t):=
\begin{cases}
\frac{\delta^\alpha}{\psi(\delta)}t^{1-\alpha}, &0\lt t\le\delta,\\
A_\delta-\log\int_t^\pi ds/\psi(s), &\delta\lt t\le\eta_\delta,\\
1, &\eta_\delta\lt t\le\pi.
\end{cases}
$$
Here $A_\delta, \eta_\delta$ are constants, chosen 
to make $w_\delta$ a continuous  function with
$0\le w_\delta\le1$.

For each $\delta$, the function $w_\delta(t)/t^{1-\alpha}$ is bounded, 
from which it follows that $f_{w_\delta}/g^{1-\alpha}$ is bounded on $\DD$. 
By Theorem~\ref{T:RS1} we deduce that $f_{w_\delta}\in[g^{1-\alpha}]$.
 
The conditions (i) and (ii) above are both easy consequences of the assertion that 
$\lim_{\delta\to0}\eta_\delta=0$, which we now prove.
Given $\epsilon\gt0$, if $\eta_\delta\gt\epsilon$, 
then $w_\delta(\epsilon)\lt1$, in other words
$$
\log\int_\delta^\pi\frac{ds}{\psi(s)}-\frac{\delta}{\psi(\delta)}-
\log\int_\epsilon^\pi\frac{ds}{\psi(s)}\lt 1.
$$
As $\delta\to0$, the left-hand side tends to infinity.
Thus $\eta_\delta\le\epsilon$ for all sufficiently small $\delta$.

We now turn to the condition (iii). 
We claim that there exists $\gamma\gt2$ such that,
for all sufficiently small $\delta\gt0$, 
the function $t\mapsto w_\delta(t^\gamma)$ is concave on $(0,\pi]$.
Assume this for the moment. Then Theorem~\ref{T:Dint} applies, 
and for all small $\delta$ we have
$$
\cD(f_{w_\delta})\asymp\int_0^\pi w_\delta'(t)^2tN_E(t)\,dt
\le \int_0^{\eta_\delta} w_{\delta}'(t)^2|E_t|\,dt.
$$
We examine this last integral separately on $(0,\delta)$ and $(\delta,\eta_\delta)$.

Let us begin with $(\delta,\eta_\delta)$. Here we have
$$
\int_\delta^{\eta_\delta}w_\delta'(t)^2|E_t|\,dt
=\int_\delta^{\eta_\delta}\Bigl(\int_t^\pi\frac{ds}{\psi(s)}\Bigr)^{-2}
\frac{|E_t|}{\psi(t)^2}\,dt.
$$
Note that if $|E_t|\le t^\beta$ then $\psi(t)\ge|E_t|$,
whereas if $|E_t|\gt t^\beta$ then $\psi(t)=t^\beta$. 
The last integral is therefore majorized by
\begin{align*}
&\int_\delta^{\eta_\delta}\Bigl(\int_t^\pi\frac{ds}{\psi(s)}\Bigr)^{-2}
\frac{1}{\psi(t)}\,dt
+\int_\delta^{\eta_\delta}\Bigl(\int_t^\pi\frac{ds}{\psi(s)}\Bigr)^{-2}
\frac{Ct^\mu}{t^{2\beta}}\,dt
\le \Bigl(\int_{\eta_\delta}^\pi\frac{ds}{\psi(s)}\Bigr)^{-1}
+C\Bigl(\int_{\eta_\delta}^\pi\frac{ds}{\psi(s)}\Bigr)^{-2} \eta_\delta^{\mu+1-2\beta},
\end{align*}
and this tends to zero as $\delta\to0$. 

Now consider what happens on $(0,\delta)$. Here we have
$$
\int_0^\delta w_\delta'(t)^2|E_t|\,dt
=\frac{\delta^{2\alpha}}{\psi(\delta)^2}\int_0^\delta t^{-2\alpha}|E_t|\,dt.
$$
If $|E_t|\le t^\beta$ for all  $t\in(0,\delta)$,
then $|E_t|/t^\alpha\le\psi(t)/t^\alpha\le\psi(\delta)/\delta^\alpha$, 
and so, 
$$
\frac{\delta^{2\alpha}}{\psi(\delta)^2}\int_0^\delta t^{-2\alpha}|E_t|\,dt
\le \frac{\delta^\alpha}{\psi(\delta)}\int_0^\delta t^{-\alpha}\,dt
=\frac{1}{1-\alpha}\frac{\delta}{\psi(\delta)}\le \frac{1}{1-\alpha}.
$$
On the other hand, if $|E_t|\gt t^\beta$ for a sequence $t=\delta_n$ tending to zero, 
then $\psi(\delta_n)=\delta_n^\beta$ for all $n$, and consequently
$$
\frac{\delta_n^{2\alpha}}{\psi(\delta_n)^2}\int_0^{\delta_n} t^{-2\alpha}|E_t|\,dt
\le \frac{\delta_n^{2\alpha}}{\delta_n^{2\beta}}\int_0^{\delta_n}t^{-2\alpha}Ct^\mu\,dt
=\frac{C}{1+\mu-2\alpha}\delta_n^{1+\mu-2\beta},
$$
which tends to zero as $n\to\infty$. Putting all of this together gives (iii).

All that remains is to establish the claim about concavity.
Fix $\gamma\gt2$ with $1-1/\gamma\lt\alpha$.
Our aim is to prove that $t^{1-1/\gamma}w_\delta'(t)$ is decreasing.
On $(0,\delta)$ we have
$$ 
t^{1-1/\gamma}w_\delta'(t)=Ct^{-\nu},
$$
where $\nu:=\alpha+1/\gamma-1\gt0$. This is certainly decreasing. 
On $(\delta,\eta_\delta)$ we have
$$
t^{1-1/\gamma}w_\delta'(t)
=\frac{t^{1-1/\gamma}}{\psi(t)}\Bigl(\int_t^\pi \frac{ds}{\psi(s)}\Bigr)^{-1}
=\frac{t^\alpha}{\psi(t)}\Bigl(t^\nu\int_t^\pi\frac{ds}{\psi(s)}\Bigr)^{-1}.
$$
Now $\psi(t)/t^\alpha$ is increasing. 
Also, the derivative of $t\mapsto t^\nu\int_t^\pi ds/\psi(s)$ 
has the same sign as 
$$ 
\nu\int_t^\pi\frac{ds}{\psi(s)}-\frac{t}{\psi(t)},
$$
which is positive if $t$ is small enough.
Thus $t^{1-1/\gamma}w_\delta'(t)$ is decreasing on $(\delta,\eta_\delta)$
provided that $\delta$ is small enough.
Lastly, at $t=\delta$, 
we need that the left derivative of $w_\delta$ exceeds the right derivative, 
which boils down to the inequality
$$ 
\int_\delta^\pi\frac{ds}{\psi(s)}\ge \frac{1}{1-\alpha},$$
and this certainly holds for small $\delta$,
since the  left-hand side tends to infinity as $\delta\to0$. In summary, 
we have shown that $t^{1-1/\gamma}w'_\delta(t)$ is decreasing on $(0,\pi]$ 
if $\delta$ is small enough. 
The claim about concavity is proved,
and with it, the theorem.\qed

\section{Bergman--Smirnov exceptional sets}\label{S:BS}

There is indirect  evidence for the  Brown--Shields conjecture 
in the form of numerous results about cyclicity in $\cD$, 
due to Brown--Shields and to others, 
all of which are consistent with the conjecture.
However, the first direct progress towards proving the conjecture 
was made by Hedenmalm and Shields in \cite{HS}, 
followed by further contributions by
Richter and Sundberg \cite{RS3} and El-Fallah, Kellay and Ransford \cite{EKR}.
In this section we  briefly describe this work 
and relate it to the results in the present paper.

Hedenmalm and Shields introduced the notion of Bergman--Smirnov exceptional set, 
which is defined as follows.
Let $\DD_e:=\{z\in\CC:|z|\gt1\}$. 
We write $\cB_e$ for the Bergman space on $\DD_e$, 
namely the holomorphic functions 
on $\DD_e$ of the form  $\sum_{k\ge0}b_k/z^{k+1}$
with $\sum_{k\ge0}|b_k|^2/(k+1)\lt\infty$. 
Also we write $\cN^+$ for the Smirnov class, 
namely the holomorphic functions on $\DD$ of the form $h_1/h_2$,
where $h_1,h_2$ are  holomorphic and bounded on $\DD$ and $h_2$ is outer.
A closed set $E\subset\TT$ is called a {\em Bergman--Smirnov exceptional set} 
(or BS-exceptional set for short) 
if it is removable for all holomorphic functions $\phi:\CC\setminus E\to\CC$
such that $\phi|\DD_e\in\cB_e$ and $\phi|\DD\in\cN^+$.

The following theorem explains the interest in BS-exceptional sets. 
It was first proved by Hedenmalm and Shields \cite{HS} 
in the case where $f$ extends continuously to $\overline{\DD}$, 
and the general case was established a little later 
by Richter and Sundberg in \cite{RS3}.

\begin{thm}\label{T:HSRS}
Let $f\in\cD$ be an outer function, 
and set $E:=\{\zeta\in\TT:\liminf_{z\to\zeta}|f(z)|=0\}$.
If $E$ is a Bergman--Smirnov exceptional set, then $f$ is cyclic.
\end{thm}

\begin{proof}
See \cite[Corollary to Theorem~1]{HS} and \cite[Corollary~3.3]{RS3}.
\end{proof}

This theorem leaves us with the problem 
of identifying exactly which sets are BS-exceptional. 
Hedenmalm and Shields proved that BS-exceptional sets 
are of capacity zero \cite[Lemma~2]{HS}, 
and they asked whether, conversely, 
every closed subset of $\TT$ of capacity zero is BS-exceptional \cite[Problem~4]{HS}. 
This problem is still open, though there are a certain number of partial results, 
which we now describe.

A closed set $E\subset\TT$ has a unique decomposition $E=E^c\cup E^p$, 
where $E^c$ is countable and $E^p$ is perfect (the {\em perfect core} of $E$).
Hedenmalm and Shields  proved that if $E^p$ is BS-exceptional 
then so is $E$ (the converse is obvious). 
In particular, since the empty set is obviously a BS-exceptional set, 
it follows that every countable compact subset of $\TT$ is BS-exceptional. 
For more on this see \cite[Theorem~3]{HS} and the remark that follows it.

Hedenmalm and Shields also showed that 
the union of two disjoint BS-exceptional sets is again BS-exceptional 
\cite[Corollary to Proposition~2]{HS}. 
It seems to be unknown whether one can relax the disjointness hypothesis. 
However, using the technique of the proof of Theorem~\ref{T:Korenblum} above, 
it is possible to show that 
the union of two BS-exceptional sets is  BS-exceptional 
if at least one of them satisfies the Carleson condition \eqref{E:Ccond}. 
We omit the details.

The first examples of  uncountable BS-exceptional sets were given in \cite{EKR}.
It was proved in \cite[Theorem~2.3]{EKR} that $E$ is BS-exceptional whenever
$$
\int_0\frac{|E_t|}{\bigl(t\log(1/t)\log\log(1/t)\bigr)^2}\,dt\lt\infty.
$$
This permits the construction of certain generalized Cantor sets that are BS-exceptional.

To these results, we can now add the following theorem.

\begin{thm}\label{T:BS}
Let $E$ be a closed subset of $\TT$ such that $|E_t|=O(t^\mu)$ for some $\mu\gt0$ and 
$$ \int_0^\pi \frac{dt}{|E_t|}=\infty.$$
Then $E$ is a Bergman--Smirnov exceptional set.
\end{thm}

\begin{proof}
It was shown in \cite[Corollary~3.2]{EKR} that $E$ is BS-exceptional 
if there exists a cyclic $f\in\cD$ satisfying $|f^*(\zeta)|\le d(\zeta,E)^2$ a.e.
Let $f$ be the outer function satisfying $|f^*|=d(\zeta,E)^2$ a.e.
By Corollary~\ref{C:powers}, applied with $w(t)=t^2$, we have $f\in\cD$.
By Theorem~\ref{T:main} $f$ is cyclic. 
\end{proof}

Using this theorem, we are able to answer the question of Hedenmalm and Shields  
at least in a special case. 
We recall that the notion of generalized Cantor set was defined in \S\ref{S:intro}.

\begin{cor}
Let $E$ be a closed subset of $\TT$ whose perfect core is a generalized Cantor set. 
Then $E$ is a Bergman--Smirnov exceptional set if and only if it is of capacity zero.
\end{cor}

\begin{proof}
The `only if' is by  \cite[Lemma~2]{HS}. As for the `if',  
Theorem~\ref{T:BS} applied to the perfect core $E^p$ shows that $E^p$ is BS-exceptional, 
from which it follows that $E$ is too.
\end{proof}


\begin{thebibliography}{99}

\bibitem{Bo}
B. Bouya,
`Id\'eaux ferm\'es de certaines alg\`ebres de fonctions analytiques',
{\em C. R. Math. Acad. Sci. Paris} {\bf 343} (2006), 235--238. 

\bibitem{BS}
L. Brown, A. Shields, 
`Cyclic vectors in the Dirichlet space',
{\em Trans. Amer. Math. Soc.} {\bf 285} (1984), 269--304.

\bibitem{Ca1}
L. Carleson,
`Sets of uniqueness for functions regular in the unit circle',
{\em Acta Math.} {\bf 87} (1952), 325--345.

\bibitem{Ca2}
L. Carleson,
`A representation formula for the Dirichlet integral',
{\em Math. Z.} {\bf 73} (1960), 190--196.

\bibitem{Ca3}
L. Carleson,
{\em Selected Problems on Exceptional Sets}, Van Nostrand, Princeton NJ, 1967.

\bibitem{EKR}
O. El-Fallah, K. Kellay, T. Ransford, 
`Cyclicity in the Dirichlet space', 
{\it Ark. Mat.} {\bf 44} (2006), 61--86.

\bibitem{Ga}
J. Garnett, 
{\em Bounded Analytic Functions}, 
revised first edition, Springer, New York, 2007.

\bibitem{HS}
H. Hedenmalm, A. Shields,
`Invariant subspaces in Banach spaces of analytic functions',
{\em Michigan Math. J.} {\bf 37} (1990), 91--104.

\bibitem{Koo}
P. Koosis, 
{\em Introduction to $H_p$ Spaces}, 
second edition, Cambridge University Press, Cambridge, 1998.

\bibitem{Kor}
B. Korenblum,
`Invariant subspaces of the shift operator in a weighted Hilbert space',
{\em Math. USSR-Sb.} {\bf 18} (1972), 111--138.

\bibitem{Ma}
A. Matheson, 
`Approximation of analytic functions satisfying a Lipschitz condition', 
{\em Michigan Math. J.} {\bf 25} (1978) 289--298.

\bibitem{RS1}
S. Richter, C. Sundberg,
`A formula for the local Dirichlet integral',
{\em Michigan Math. J.} {\bf 38} (1991), 355--379.

\bibitem{RS2}
S. Richter, C. Sundberg,
`Multipliers and invariant subspaces in the Dirichlet space',
{\em J. Operator Theory} {\bf 28} (1992), 167--186.

\bibitem{RS3}
S. Richter, C. Sundberg,
`Invariant subspaces of the Dirichlet shift and pseudocontinuations',
{\em Trans. Amer. Math. Soc.} {\bf 341} (1994), 863--879.

\bibitem{Ro}
W. T. Ross,
`The classical Dirichlet space',
Recent advances in operator-related function theory,  171--197, 
{\em Contemp. Math.} {\bf 393}, 
Amer. Math. Soc., Providence, RI, 2006.



\end{thebibliography}
\end{document}